\documentclass[papera4,twoside]{article}
\usepackage{amssymb}
\usepackage{amsmath}
\usepackage{graphicx}
\usepackage{fancyhdr}
\usepackage{t1enc}
\usepackage[english]{babel}

\newtheorem{theorem}{Theorem}

\newtheorem{corollary}[theorem]{Corollary}

\newtheorem{definition}[theorem]{Definition}

\newtheorem{lemma}[theorem]{Lemma}

\newtheorem{remark}[theorem]{Remark}

\newenvironment{proof}[1][Proof]{\textbf{#1.} }{\ \rule{0.5em}{0.5em} \bigskip}
\renewenvironment{abstract}{\begin{center}
\begin{minipage}[c]{12cm} {\bf Abstract:}} {\end{minipage}
\end{center}}

\begin{document}

\title{The second Noether theorem on time scales}

\author{
Agnieszka B. Malinowska \footnote{Faculty of Computer Science,
Bialystok University of Technology,
15-351 Bia\l ystok, Poland, a.malinowska@pb.edu.pl} \ \& Nat\'{a}lia Martins \footnote{CIDMA - Center for Research and Development in Mathematics and Applications,
Department of Mathematics,
University of Aveiro,
3810-193 Aveiro, Portugal, natalia@ua.pt}}
\date{}

\maketitle


\begin{abstract}
We extend the second Noether theorem to variational problems on time scales. Our result provides as corollaries the
classical second Noether theorem, the second Noether theorem for the $h$-calculus and the second Noether theorem
for the $q$-calculus.

\bigskip


\bigskip


\end{abstract}


\section{Introduction}

In 1915, general relativity was almost a finished theory, but still there was a problem regarding the conservation of energy.
David Hilbert asked for help in solving this problem the German mathematician Emmy Noether.
She solved the problem proving two remarkable theorems that relate the invariance of a variational integral with
properties of its Euler--Lagrange equations. These results were published  in 1918 in the paper \emph{Invariante Variationsprobleme} \cite{Noether1918}.
Noether was described by many important scientists, such as Pavel Alexandrov, Albert Einstein, Jean Dieudonné and David Hilbert, as the most important woman in the history of mathematics.
In order to get a good exposition of the history of Emmy Noether and her important contributions to fundamental physics and mathematics, we refer the reader
to the recent book \cite{Kosmann-Schwarzbach2010}.
This book explains very clearly that it took too much time before mathematicians and physicists began to recognize the
importance of Noether's theorems: until 1950 Noether's theorems were poorly understood and Noether's name disappeared almost entirely.

The first theorem in \cite{Noether1918}, usually known as Noether's theorem,  guarantees that the
invariance of a variational
integral with respect to a continuous symmetry transformations that depend on $\rho$ parameters implies the existence of $\rho$ conserved quantities along the
Euler--Lagrange extremals. Such transformations are global transformations.
Noether's theorem explains all conservation laws of mechanics:  conservation of energy comes
from invariance of the system under time translations;
conservation of linear momentum comes from invariance of the system under spacial translations; conservation of angular
momentum reflects invariance with respect to spatial rotations.

The first Noether theorem is nowadays a well-known tool in a modern theoretical physics, engineering and the calculus of variations \cite{Atanackovic}, \cite{Zbig+Torres 2008}, \cite{Gastão+Torres 2007}, \cite{FredericoTorres2012}, \cite{MalinowskaNother2012}, \cite{bookMaliTorres2012}, \cite{MartinsTorres2010}.
Inexplicably, it is still not well-known that the famous paper of
Emmy Noether includes another important result, the second Noether theorem,
which applies to variational
problems that are invariant under a certain group of transformations, a so-called infinite continuous group,
which depends on arbitrary functions and their derivatives (see also \cite{Logan1974}).
Such transformations are local transformations because can affect every part of the system differently.
Noether's second theorem states that if a variational integral has an infinite-dimensional Lie algebra of
infinitesimal symmetries parameterized linearly by $r$ arbitrary functions and their derivatives up to a given order $m$,
then there are $r$ identities between Euler--Lagrange expressions and their derivatives up to order $m$. These identities Noether called ``dependencies''. For example, the Bianchi identities, in the general theory of relativity, are examples of such ``dependencies''.
Noether's second theorem has applications in general relativity, electrodynamics, hydromechanics,
quantum chromodynamics and other gauge field theories. Motivated by the important
applications of the second Noether theorem, our goal in this paper is to generalize this result
proving that the second Noether Theorem is valid
for an infinite number of time scales.
As we will see, in the particular case where the time scale $\mathbb{T}$ is $\mathbb{R}$, we get from our result the classical second
Noether theorem; when $\mathbb{T}=\mathbb{Z}$ we obtain the analogue of the second Noether theorem for the difference calculus of variations; when $\mathbb{T}=q^{\mathbb{N}_0}$ (for some $q>1$) we obtain a new result: the second Noether theorem for the $q$-calculus (quantum calculus). For more on the theory of quantum calculus and quantum calculus of variation we refer to \cite{Aldwoah+Mali+Torres2012}, \cite{Bangerezako 2004}, \cite{Baoguo et all 2009}, \cite{Brito+Martins+Torres2012}, \cite{Brito+Martins+Torres q-symmetric}, \cite{Brito+Martins+Torres Hahn-symmetric}, \cite{Cresson 2009}, \cite{Mali+Martins 2013}, \cite{MaliTorresHahn2010}, \cite{MartinsTorres Infinite 2012}, \cite{Book Kac 2002}.

The theory of time scales was introduced in 1988 by Stefan Hilger in his Ph.D thesis \cite{Hilger} as a means of unifying
theories of differential calculus and difference calculus into a single theory. With a short time this unification aspect has been supplemented by the extension and generalization features. The time scale calculus allows to consider more complex time domains, such as $h\mathbb{Z}$, $\mathbb{T}=q^{\mathbb{N}_0}$ or hybrid domains.
The study of the calculus of variations in the context of time scales
had its beginning only in 2004 with the paper \cite{Bohner:2004} of
Martin Bohner (see also \cite{zeidan:2004}). Since then, the variational calculus on time scales
advanced fairly quickly, as can be verified with the large number of published papers on the subject
\cite{Almeida and Torres(2009)},  \cite{RuiTorres2010}, \cite{Rui+Delfim2008},
 \cite{FerreiraTorres2010}, \cite{EwaMaliTorres2010},  \cite{MMT-2010}, \cite{MaliTorresAMC2010},
\cite{comBasia:duality}, \cite{Martins+Torres-2009},
\cite{NataliaTorresGeneralizing-2011}, \cite{MartinsTorres2012}.
Noether's first theorem has been
extended to the variational calculus on time scales using several approaches \cite{Zbig+Torres 2008}, \cite{Zbig+Nat+Torres2011}, \cite{MartinsTorres2010}, while the second Noether theorem on times scales is still not available in the literature. So there is evidently a need for a time scale analogue of Noether's second theorem.

The paper is organized as follows. In Section~\ref{Preliminary results} we
review some preliminaries about single-variable variational calculus on time scales, for example we recall the Euler--Lagrange equation for
a delta variational problem. Our main results are stated in Section~\ref{Main results}. Namely, in Subsection~\ref{Second Noether's Theorem - single integral case} we prove
Noether's second theorem for variational problems involving a single delta integral (with and without transformation of time) and in Subsection~\ref{Second Noether's Theorem - multiple integral case} we prove
Noether's second theorem for variational problems involving multiple delta integrals (without transformation of time). Section~\ref{Example}
provides a concrete example of application of our results. Finally, in Section~\ref{Concluding Remarks}
we present some concluding remarks.


\section{Preliminaries}
\label{Preliminary results}

In this paper we assume the reader to be familiar with the calculus on time scales.
For a good introduction to the theory of time scales we refer
to the well-known books in this field \cite{Bohner-Peterson1,Bohner-Peterson2}.
The first developments on time scale calculus was done essentially
using the delta-calculus. However, for some applications, in particular to solve problems of
the calculus of variations and control theory in economics, is often more convenient to
work backwards in time, that is, using the nabla-calculus.
In this paper we are concerned with the delta-calculus. It is clear that all the arguments used in the
proofs of our results can be modified to work for the nabla-calculus.

In what follows we review some preliminaries about the
variational calculus on time scales needed in this paper.

Let $\mathbb{T}$ be a given time
scale, $n \in \mathbb{N}$,
and $L: \mathbb{T}\times \mathbb{R}^n \times \mathbb{R}^n
\rightarrow \mathbb{R}$ be continuous, together with its partial delta derivatives of first and second order with respect to $t$
and partial usual derivatives of the first and second order with respect to the other variables.
Suppose that $a,b\in \mathbb{T}$ and $a<b$.
We consider the following optimization problem on $\mathbb{T}$:
\begin{equation}
\label{problem}
\mathcal{L}[y]=\int_a^b L(t,y^\sigma(t),y^\Delta(t)) \Delta t
\longrightarrow {\rm extremize},
\end{equation}
where the set of admissible functions are
$$
\mathcal{D}=\{ y\ | \ y: [a,b]\cap \mathbb{T}
\rightarrow \mathbb{R}^n,\ y\in \mathrm{C}^1_{rd}([a,b]\cap \mathbb{T},\mathbb{R}^n) ,\
y(a)=\alpha,\ y(b)=\beta\}
$$
for some $\alpha, \beta \in \mathbb{R}^n$, and
where $\sigma$ is the forward jump operator,  $y^\Delta$
is the delta-derivative of $y$, and, for $i\in \mathbb{N}$,
$$
\mathrm{C}^i_{rd}([a,b]\cap \mathbb{T},\mathbb{R}^n):= \{ y\ | \ y: [a,b]\cap \mathbb{T}
\rightarrow \mathbb{R}^n,\ y^{\Delta^i} \mbox{ is rd-continuous on } [a,b]^{\kappa^i}\}.
$$

As usual, $y^\sigma(t)$ denotes  $y(\sigma(t))$ and $[a,b]^{\kappa^i}:=[a,\rho^i(b)]$, where $\rho$ is the backward jump operator.
By extremize we mean maximize or minimize.


In what follows all intervals are time scales intervals, that is, we simply write $[a,b]$  to denote the set
$[a,b]\cap \mathbb{T}$.
Let $y:=(y_1, \ldots, y_n)$ and denote by $\frac{\partial L}{\partial y_k}$
the partial derivative of $L$ with respect to $y_k$.

\begin{definition}
We say that $y_{\ast}\in C_{rd}^{1}([a,b], \mathbb{R}^n)$
is a local minimizer (resp. local maximizer) for problem  (\ref{problem})
if there exists $\delta > 0$ such that
$$
\mathcal{L}[y_{\ast}]\leq \mathcal{L}[y] \quad ( {\rm resp.}  \ \mathcal{L}[y_{\ast}]\geq \mathcal{L}[y])
$$
for all $y \in C_{rd}^{1}([a,b], \mathbb{R}^n)$
satisfying the boundary conditions
$y(a)=\alpha$, $y(b)=\beta$, and
$$
\parallel y - y_{\ast}\parallel :=
\sup_{t \in [a,b]^\kappa}\mid y^{\sigma}(t)-y_{\ast}^{\sigma}(t)\mid
+ \sup_{t \in [a,b]^\kappa}\mid y^{\Delta}(t)-y_{\ast}^{\Delta}(t)\mid < \delta \, ,
$$
where $|\cdot|$ denotes a norm in $\mathbb{R}^n$.
\end{definition}

\begin{definition}
We say that $\eta\in C^{1}_{rd}\left([a,b],\mathbb{R}^n\right)  $ is an admissible variation for problem (\ref{problem}) provided
$\eta\left(  a\right)  =\eta\left(  b\right)  =0.$
\end{definition}

\begin{definition}
A function $f:[a,b]\times \mathbb{R} \rightarrow \mathbb{R}$ is called continuous in the second variable,
uniformly in the first variable, if for each $\epsilon >0$, there exists $\delta>0$ such that $|x_1-x_2| < \delta$ implies
$|f(t,x_1)- f(t,x_2)| < \epsilon$ for all $t \in [a,b]$.
\end{definition}

\begin{lemma} [\cite{Bohner:2004}] \label{derivation under integral sign} Suppose that $\eta:=(\eta_1, \cdots, \eta_n)$ is an admissible variation for problem (\ref{problem}) and $y:=(y_1, \cdots, y_n) \in \mathcal{D}$.
Let $\phi:\mathbb{R}\rightarrow \mathbb{R}$ and $f:[a,b]\times \mathbb{R} \rightarrow \mathbb{R}$  be defined, respectively, by
$\phi(\epsilon):=\mathcal{L}[y+\epsilon \eta]$ and
$f(t,\epsilon):=L(t,y^\sigma(t) + \epsilon \eta^\sigma(t),y^\Delta(t) + \epsilon \eta^\Delta (t))$.
If $\frac{\partial f}{\partial \epsilon}$ is continuous in $\epsilon$, uniformly in $t$, then
$$
\dot{\phi}(0)= \displaystyle \int_a^b \sum_{k=1}^{n}\left(\frac{\partial L}{\partial y^\sigma_k} (t,y^\sigma(t),y^\Delta(t))\eta_k^\sigma(t) + \frac{\partial L}{\partial y^\Delta_k} (t,y^\sigma(t),y^\Delta(t))\eta_k^\Delta(t)\right) \Delta t.
$$
\end{lemma}

Next we present the following result that is a fundamental tool in the calculus of variations on time scales.

\begin{theorem} [Euler--Lagrange equation on time scales, \cite{Bohner:2004}]  \label{Euler--Lagrange: Bohner} If $y_\ast$ 
is a weak local extremizer for problem (\ref{problem}) and $L$ satisfies the assumption of Lemma \ref{derivation under integral sign}, for every $y$ and $\eta$,
then the components of $y_\ast$ satisfies the $n$
Euler--Lagrange equations
\begin{equation*}
\frac{\Delta}{\Delta t}\frac{\partial L}{\partial y^\Delta_k}(t,y^\sigma_\ast(t),y_\ast^\Delta(t))
=\frac{\partial L}{\partial y^\sigma_k}(t,y^\sigma_\ast(t),y_\ast^\Delta(t)), \quad \quad k=1,\ldots, n
\end{equation*}
for $t\in[a,b]^{\kappa}$.
\end{theorem}

It is well-known that the forward jump operator, $\sigma$, is not delta differentiable for certain time scales.
Also, the chain rule as we know it from the classical calculus (that is, when $\mathbb{T}=\mathbb{R}$) is not valid in general.
However, if we suppose that the time scale $\mathbb{T}$ satisfies the following condition

\begin{description}
\item[$(H)$] \quad \quad for each $t \in\mathbb{T}$, $\sigma(t)= b_1 t +  b_0$ for some
$b_1\in\mathbb{R}^+$ and $b_0\in\mathbb{R}$
\end{description}

\noindent then we can deal with these two limitations as noticed in Remark \ref{rem:rest:H} and Lemma
\ref{lemmaderivadacomposta}.

\begin{remark}
\label{rem:rest:H}
Note that
condition $(H)$ implies that $\sigma$ is delta
differentiable and
$\sigma^{\Delta}(t)=b_1$, $t \in \mathbb{T}^{\kappa}$. Also note that
condition $(H)$ describes, in particular,  the differential calculus
($\mathbb{T}=\mathbb{R}$, $b_1=1$, $b_0=0$), the difference
calculus ($\mathbb{T}=\mathbb{Z}$, $b_1=1$, $b_0=1$), the $h$-calculus
($\mathbb{T}=h \mathbb{Z}:=\{hz : z \in \mathbb{Z}\}$, $b_1=1$, $b_0=h$ for some $h>0$),
and the
$q$-calculus ($\mathbb{T}=q^{\mathbb{N}_0}:=\{ q^k: k \in
\mathbb{N}_0\}$ for some $q>1$, $b_1=q$, $b_0=0$).
\end{remark}

\begin{lemma}[\cite{Rui+Delfim2008}]
\label{lemmaderivadacomposta} Let $\mathbb{T}$ be a time scale satisfying condition $(H)$. If
$f:\mathbb{T}\rightarrow \mathbb{R}$ is two times delta
differentiable, then
$$
\label{Condition-lemma2} f^{\sigma\Delta}(t)=b_1 f^{\Delta\sigma}(t)\, ,
\quad t\in\mathbb{T}^{\kappa^2}.
$$
\end{lemma}

The next result is also useful for the proofs of our main results.

\begin{lemma}[cf. \cite{MartinsTorres2012}] \label{lemma0}
Assume that the time scale $\mathbb{T}$ satisfies condition $(H)$, $m \in \mathbb{N}$,
and $\eta \in C^{2m}_{rd}([a,b], \mathbb{R})$ is such that
$\eta^{\Delta ^{i}}(a)=0$ for all $i=0,1,\ldots, m$.
Then, $\eta^{\sigma \Delta^{i-1}}(a)=0$ for each $i=1,\ldots, m$.
\end{lemma}


\section{Main results}
\label{Main results}

In this section we formulate and prove the second Noether theorem for single and multiple integral variational problems.

\subsection{Noether's second theorem - single delta integral case}
\label{Second Noether's Theorem - single integral case}

In this subsection we suppose that the time scale $\mathbb{T}$ satisfies condition (H) and that
$L$ satisfies the assumption of Lemma \ref{derivation under integral sign}, for every $y$ and $\eta$.
As usual, $\eta^{\Delta^0}$ and $\eta^{\sigma^0}$ denote $\eta$.
Let $m$ be a fixed natural number.
We also assume that the time scale interval $[a,b]$ has, at least, $2m+1$ points.

We begin with some technical results that will be useful in the proofs of Theorems \ref{Theorem without transformation of time}
and \ref{Theorem with transformation of time}.

\begin{lemma} [Higher-order fundamental lemma of the calculus of variations] \label{higher-orderFundamentalLemma}
Let $\mathbb{T}$ be a time scale satisfying condition (H)
and  $f_0, f_1, \cdots, f_m \in C_{rd}([a,b], \mathbb{R})$. If

$$\int_{a}^{\rho^{m-1}(b)}
\left(\sum_{i=0}^{m}f_i(t) \eta^{\sigma^{m-i}\Delta^{i}}(t) \right)
\Delta t=0$$

for all $\eta \in C_{rd}^{2m}([a, b], \mathbb{R})$ such that
\begin{align}
\eta\left(a\right)=0,& \quad \eta\left(\rho^{m-1}(b)\right)=0, \nonumber\\
&\vdots\nonumber \\
\eta^{\Delta^{m-1}}\left(a\right)=0,&\quad
\eta^{\Delta^{m-1}}\left(\rho^{m-1}(b)\right)=0,\nonumber
\end{align}

then

$$\sum_{i=0}^{m} (-1)^i \left(\frac{1}{b_1}\right)^{\frac{i(i-1)}{2}}f_i^{\Delta^i}(t) =0 \ ,  \ \ \ \ t \in [a,b]^{\kappa^{m}}.$$
\end{lemma}

\begin{proof} The proof is similar to the proof of Lemma 16 of  \cite{Martins+Torres-2009}.
\end{proof}

\begin{remark}  We emphasize that the delta differentiability of the
functions $f_0, f_1, \cdots, f_m$ was not assumed in advance.
\end{remark}


\begin{lemma} \label{lemma1}
Assume that the time scale $\mathbb{T}$ satisfies condition (H) and $\eta \in C_{rd}^{2m}([a,b],\mathbb{R})$ is such that,
$$
\eta ^{\Delta^i}(a)=0, \quad \quad i=0,1, \ldots, m.
$$
Then,
$$
\eta ^{\sigma^i}(a)=0, \quad \quad i=0,1, \ldots, m.
$$
\end{lemma}

\begin{proof}
If $a$ is right-dense, the result is trivial. Suppose that $a$ is right-scattered.
Since
$\displaystyle \eta^{\Delta}(a)=\frac{\eta^{\sigma}(a)-\eta(a)}{\sigma(a)-a}=0$ and $\eta(a)=0$, we conclude that $\eta^\sigma(a)=0$.
Since
$\displaystyle  \eta^{\Delta^2}(a)=(\eta^{\Delta})^{\Delta}(a)=\frac{(\eta^{\Delta})^{\sigma}(a)-\eta^{\Delta}(a)}{\sigma(a)-a}=0$
and $\eta^{\Delta}(a)=0$,
then
$(\eta^{\Delta})^{\sigma}(a)=0$.
Using Lemma \ref{lemmaderivadacomposta}, we get $(\eta^{\sigma})^{\Delta}(a)=0$.
Since
$
\displaystyle (\eta^{\sigma})^{\Delta}(a)= \frac{\eta^{\sigma^2}(a)-\eta^{\sigma}(a)}{\sigma(a)-a}=0
$
and
$\eta^{\sigma}(a)=0$, we conclude that $\eta^{\sigma^2}(a)=0$.

Since
$\displaystyle  \eta^{\Delta^3}(a)=(\eta^{\Delta^2})^{\Delta}(a)=\frac{(\eta^{\Delta^2})^{\sigma}(a)-\eta^{\Delta^2}(a)}{\sigma(a)-a}=0$
and $\eta^{\Delta^2}(a)=0$ ,
then
$\eta^{\Delta^2\sigma}(a)=0$.
Using Lemma \ref{lemmaderivadacomposta}, we get $(\eta^{\sigma})^{\Delta^2}(a)=0$.
Since
$\displaystyle (\eta^{\sigma})^{\Delta^2}(a) = \frac{\eta^{\sigma\Delta\sigma}(a) - \eta^{\sigma\Delta}(a)}{\sigma(a)-a}=0$ and $\eta^{\sigma\Delta}(a)=0$, then
$\eta^{\sigma \Delta \sigma}(a)=0$. Lemma \ref{lemmaderivadacomposta} proves that $\eta^{\sigma^2 \Delta}(a)=0$. Since,
$
\displaystyle \eta^{\sigma^2 \Delta}(a)= \frac{\eta^{\sigma^3}(a) - \eta^{\sigma^2}(a)}{\sigma(a)-a}=0
$
and
$\eta^{\sigma^2}(a)=0$,
we get
$\eta^{\sigma^3}(a)=0$.
Repeating recursively this process, we conclude the proof.
\end{proof}

\begin{lemma} \label{lemma2}
Assume that the time scale $\mathbb{T}$ satisfies condition (H) and $\eta \in C_{rd}^{2(m-1)}([a,b],\mathbb{R})$ is such that,
$$
\eta ^{\Delta^i}(a)=0, \quad \quad i=0,1, \ldots, m-1.
$$
Then,
$$
\eta^{\sigma\Delta^{m-2}}(a)= \eta^{\sigma^2\Delta^{m-3}}(a) = \eta^{\sigma^3\Delta^{m-4}}(a)= \ldots = \eta^{\sigma^{m-2}\Delta}(a)= 0.
$$
\end{lemma}

\begin{proof}
If $a$ is right-dense, the result is trivial. Suppose that $a$ is right-scattered.
Since $\displaystyle  \eta^{\Delta^{m-1}}(a)=\frac{\eta^{\Delta^{m-2}\sigma}(a) - \eta^{\Delta^{m-2}}(a)}{\sigma(a)-a}=0$ and $\eta^{\Delta^{m-2}}(a)=0$,
then $\eta^{\Delta^{m-2}\sigma}(a)=0$. Using Lemma \ref{lemmaderivadacomposta}, we get $\eta^{\sigma\Delta^{m-2}}(a)=0$ (or use Lemma \ref{lemma0}).
Note that
$\displaystyle \eta^{\sigma\Delta^{m-2}}(a)= \frac{\eta^{\sigma\Delta^{m-3}\sigma}(a) - \eta^{\sigma\Delta^{m-3}}(a)}{\sigma(a)-a}=0$.
Lemma \ref{lemma0} shows that
$\eta^{\sigma\Delta^{m-3}}(a)=0$, hence
$\eta^{\sigma\Delta^{m-3}\sigma}(a)=0$. Using Lemma \ref{lemmaderivadacomposta} we get
$\eta^{\sigma^2\Delta^{m-3}}(a)=0$.
Next we prove that $\eta^{\sigma^2\Delta^{m-4}}(a)=0$.
Since $\displaystyle \eta^{\sigma\Delta^{m-3}}(a)= \frac{\eta^{\sigma\Delta^{m-4}\sigma}(a) - \eta^{\sigma\Delta^{m-4}}(a)}{\sigma(a)-a}=0$ and
$\eta^{\sigma\Delta^{m-4}}(a)=0$ (by Lemma \ref{lemma0}), then $\eta^{\sigma\Delta^{m-4}\sigma}(a)=0$. Lemma \ref{lemmaderivadacomposta}
shows that $\eta^{\sigma^2\Delta^{m-4}}(a)=0$.
Since
$\displaystyle \eta^{\sigma^2\Delta^{m-3}}(a)=  \frac{\eta^{\sigma^2\Delta^{m-4}\sigma}(a) - \eta^{\sigma^2\Delta^{m-4}}(a)}{\sigma(a)-a}=0$
and
$\eta^{\sigma^2\Delta^{m-4}}(a)=0$, then $\eta^{\sigma^2\Delta^{m-4}\sigma}(a)=0$. Lemma \ref{lemmaderivadacomposta} shows that
$\eta^{\sigma^3\Delta^{m-4}}(a)=0$.
Repeating recursively this process, we prove the intended result.
\end{proof}

Let $y:=(y_1, y_2, \cdots, y_n)$.
Firstly we will prove the second Noether theorem without transformation of time.
For that consider the following transformations that depend on arbitrary functions $p_1, p_2, \ldots, p_r$ and their delta-derivatives up to order $m$:

\begin{equation}\label{without transformations of time}
\left\{
\begin{array}{ccl}
\overline{t}  & = & t\\
& & \\
\overline{y}_k(t)& = & y_k(t) + \displaystyle \sum_{j=1}^{r}T^{kj}(p_{j})(t), \quad k= 1,2,\ldots, n
\end{array}
\right.
\end{equation}
\\

where, for each $j=1,2,\ldots, r$, 
$p_{j} \in C^{2m}_{rd}([a,\sigma^{m}(b)], \mathbb{R})$,

$$
T^{kj}(p_{j}):= \sum_{i=0}^m g^{k}_{ij}p^{\sigma^{m-(i+1)} \Delta^i}_{j}
$$
and $g^{k}_{ij} \in C^{1}_{rd}([a,b], \mathbb{R})$.

\begin{definition}\label{invariance y}
Functional $\mathcal{L}$ is invariant under transformations (\ref{without transformations of time})
if, and only if, for all $y\in C^1_{rd}([a,b],\mathbb{R})$ we have
$$
\int_a^b L(t,y^\sigma(t),y^\Delta(t)) \Delta t = \int_a^b L(t,\overline{y}^\sigma(t),\overline{y}^\Delta(t)) \Delta t.
$$
\end{definition}

\begin{remark} Note that the most common  definition of invariance (with equality of the integrals for
any subinterval $[t_a,t_b]\subseteq [a,b]$ with $a,b \in \mathbb{T}$)
implies Definition \ref{invariance y}.
\end{remark}

\begin{theorem}[Necessary condition of invariance]\label{necessary condition for invariance}
If functional $\mathcal{L}$ is invariant under transformations (\ref{without transformations of time}), then
\begin{equation}\label{invariance}
\displaystyle \sum_{k=1}^{n} \int_a^b  \left(
\frac{\partial L}{\partial y^\sigma_k} (t,y^\sigma(t),y^\Delta(t))
 \cdot \left( \sum_{j=1}^{r} T^{kj}(p_{j}) \right)^\sigma (t)
 + \frac{\partial L}{\partial y^\Delta_k} (t,y^\sigma(t),y^\Delta(t)) \cdot
\left( \sum_{j=1}^{r} T^{kj}(p_{j}) \right)^\Delta (t)
\right) \Delta t=0.
\end{equation}
\end{theorem}

\begin{proof}
Using the definition of invariance and noting that the family of transformations
(\ref{without transformations of time}) depend upon arbitrary functions
$p_1, p_2, \ldots, p_r$, we conclude that,
for any real number $\varepsilon$,

$
\begin{array}{lc}
\int_a^b L(t,y^\sigma(t),y^\Delta(t)) \Delta t &=
\int_a^b L\left(t,y_1^\sigma(t)+\varepsilon \left( \sum_{j=1}^{r} T^{1j}(p_{j}) \right)^\sigma (t), \ldots,
y_n^\sigma(t)+\varepsilon \left( \sum_{j=1}^{r} T^{nj}(p_{j}) \right)^\sigma (t)\right., \\
& \left.y_1^\Delta(t)+\varepsilon \left( \sum_{j=1}^{r} T^{1j}(p_{j}) \right)^\Delta (t), \ldots,
y_n^\Delta(t)+\varepsilon \left( \sum_{j=1}^{r} T^{nj}(p_{j}) \right)^\Delta (t)\right)
\Delta t.
\end{array}
$
Differentiating with respect to $\varepsilon$ (use Lemma \ref{derivation under integral sign}) and taking $\varepsilon=0$, we get equality (\ref{invariance}).
\end{proof}

Define

$$
E_k(L):=  \frac{\partial L}{\partial y^\sigma_k} - \frac{\Delta}{\Delta t}\frac{\partial L}{\partial y^\Delta_k}, \quad \quad k=1,2, \ldots, n.
$$

We shall call $E_k(L)$, $k=1,2, \ldots, n$, the Euler--Lagrange expressions associated to the Lagrangian $L$.

\begin{theorem}[Noether's second theorem without transforming time]\label{Theorem without transformation of time}
If functional $\mathcal{L}$ is invariant under transformations (\ref{without transformations of time}), then
there exist the following identities
$$
\sum_{k=1}^{n} \sum_{i=0}^{m} (-1)^i \left( \frac{1}{b_1} \right)^{\frac{i(i+1)}{2}} \left((g^{k}_{ij})^{\sigma} E_k(L) \right)^{\Delta^i}\equiv 0, \quad \quad  j=1,2,\ldots, r.
$$

\end{theorem}

\begin{proof}
Using the necessary condition of invariance (Theorem \ref{necessary condition for invariance}) we conclude that
$$
\displaystyle \sum_{k=1}^{n} \int_a^b  \left(
\frac{\partial L}{\partial y^\sigma_k} (t,y^\sigma(t),y^\Delta(t))
 \cdot \left( \sum_{j=1}^{r} T^{kj}(p_{j}) \right)^\sigma (t)
 + \frac{\partial L}{\partial y^\Delta_k} (t,y^\sigma(t),y^\Delta(t)) \cdot
\left( \sum_{j=1}^{r} T^{kj}(p_{j})\right)^\Delta (t)
\right) \Delta t=0.
$$

Fix   $j=1,2,\ldots, r$. By the arbitrariness of $p_1, p_2, \ldots, p_r$ we
can suppose  that $p_{h}\equiv 0$ for  $h\neq j$.
Therefore
$$
\displaystyle \sum_{k=1}^{n} \int_a^b \Big(
 \frac{\partial L}{\partial y^\sigma_k} (t,y^\sigma(t),y^\Delta(t))
 \cdot \left(T^{kj}(p_{j}) \right)^\sigma (t) +  \frac{\partial L}{\partial y^\Delta_k} (t,y^\sigma(t),y^\Delta(t)) \cdot
\left( T^{kj}(p_{j}) \right)^\Delta (t)
\Big) \Delta t=0.
$$
Integrating by parts we obtain
$$
\displaystyle  \sum_{k=1}^{n} \Big(\int_a^b \Big(
 \frac{\partial L}{\partial y^\sigma_k} (t,y^\sigma(t),y^\Delta(t)) -
 \frac{\Delta}{\Delta t}\frac{\partial L}{\partial y^\Delta_k} (t,y^\sigma(t),y^\Delta(t)) \Big)\cdot
 \left(T^{kj}(p_{j}) \right)^\sigma (t)
 \Delta t
$$
$$
  +  \Big[ \frac{\partial L}{\partial y^\Delta_k} (t,y^\sigma(t),y^\Delta(t))\cdot T^{kj}(p_{j})(t) \Big]^b_a \Big) =0.
$$

By the arbitrariness of
$p_{j}$  we can restrict to those functions such that

\begin{align*}
p_{j}\left(a\right)=0,& \quad p_{j}\left(b\right)=0, \nonumber\\
&\vdots\nonumber \\
p_{j}^{\Delta^{m-1}}\left(a\right)=0,&\quad
p_{j}^{\Delta^{m-1}}\left(b\right)=0,\nonumber
\end{align*}
and
\begin{align*}
p_{j}^{\sigma^{-1}\Delta^m}\left(a\right)=0,& \quad p_{j}^{\sigma^{-1}\Delta^m}\left(b\right)=0. \nonumber\\
\end{align*}

Using Lemmas \ref{lemma0}, \ref{lemma1}, and \ref{lemma2} we conclude that
$ T^{kj}(p_{j})(a)=0$ and $ T^{kj}(p_{j})(b)=0$, $k=1,2, \ldots, n$.
Then
$$
\displaystyle  \sum_{k=1}^{n} \int_a^b \Big(
 \frac{\partial L}{\partial y^\sigma_k} (t,y^\sigma(t),y^\Delta(t)) -
 \frac{\Delta}{\Delta t}\frac{\partial L}{\partial y^\Delta_k} (t,y^\sigma(t),y^\Delta(t)) \Big)\cdot
 \left(T^{kj}(p_{j}) \right)^\sigma (t)
 \Delta t=0
$$
that is,

$$
\displaystyle  \sum_{k=1}^{n} \int_a^b E_k(L)(t,y^\sigma(t),y^\Delta(t)) \cdot \left(T^{kj}(p_{j}) \right)^\sigma (t)
 \Delta t=0.
$$
Hence,
$$
\displaystyle \sum_{k=1}^{n} \int_a^b E_k(L)(t,y^\sigma(t),y^\Delta(t))  \cdot \left( \sum_{i=0}^{m} g^{k}_{ij}p^{\sigma^{m-(i+1)} \Delta^i}_{j} \right)^\sigma
(t) \Delta t=0.
$$
Therefore, by Lemma \ref{lemmaderivadacomposta}, we get
$$
 \int_a^b \sum_{i=0}^{m} \displaystyle \sum_{k=1}^{n} E_k(L)(t,y^\sigma(t),y^\Delta(t))\cdot  (g^{k}_{ij})^\sigma(t) \left(\frac{1}{b_1}\right)^i p^{\sigma^{m-i} \Delta^i}_{j}(t)
 \Delta t=0.
$$
Applying the higher-order fundamental lemma of the calculus of variations (Lemma \ref{higher-orderFundamentalLemma}) we obtain
$$
\displaystyle   \sum_{i=0}^{m}\sum_{k=1}^{n} (-1)^i \left(\frac{1}{b_1}\right)^{\frac{i(i-1)}{2}}  \left( E_k(L)(g^{k}_{ij})^\sigma  \left(\frac{1}{b_1}\right)^i \right)^{\Delta^i} \equiv 0
$$
which is equivalent to
$$
\displaystyle \sum_{k=1}^{n}  \sum_{i=0}^{m}(-1)^i \left(\frac{1}{b_1}\right)^{\frac{i(i+1)}{2}}  \left((g^{k}_{ij})^\sigma  E_k(L)  \right)^{\Delta^i}\equiv 0
$$
proving the desired result.
\end{proof}

We present some particular results that follow from Theorem \ref{Theorem without transformation of time} in the case where $\mathbb{T}=\mathbb{R}$,
$\mathbb{T}=h \mathbb{Z}$ (for some $h>0$), and $\mathbb{T}=q^{\mathbb{N}_0}$ (for some $q>1$).
If $\mathbb{T}=\mathbb{R}$, then $\sigma(t)=t$, $f^\Delta (t)=\dot{f}(t)$ and we obtain the classical second Noether theorem without transformation of time.

\begin{corollary}[cf. \cite{Noether1918}] 
Let $L: \mathbb{R}\times \mathbb{R}^n \times \mathbb{R}^n
\rightarrow \mathbb{R}$ be a $C^2$ function.
If functional $\mathcal{L}$ defined by
$$
\mathcal{L}[y]=\int_a^b L(t,y(t),\dot{y}(t))\, dt
$$
is invariant under transformations (\ref{without transformations of time}) (where $\sigma$ denotes in this context the identity function and $\Delta$ denotes the usual derivative),
then
there exist the following identities
$$
\sum_{k=1}^{n} \sum_{i=0}^{m} (-1)^i
\left[
g^{k}_{ij} E_k(L) \right]^{(i)}
\equiv 0,
\quad\quad j=1,2,\ldots, r.
$$
\end{corollary}

Choosing $h \mathbb{Z}$ we obtain  Noether's second  theorem without transformation of time for the $h$-calculus.

\begin{corollary} 
Let $h>0$, $L: h \mathbb{Z}\times \mathbb{R}^n \times \mathbb{R}^n
\rightarrow \mathbb{R}$ and $a,b \in h \mathbb{Z}$, $a<b$.
If functional $\mathcal{L}$ defined by
$$
\mathcal{L}[y]=\sum_{t=a}^{b-h} L(t,y(t+h),\Delta_h[y](t))
$$
is invariant under transformations (\ref{without transformations of time}) (where $\sigma$ denotes in this context the function $\sigma(t)=t+h$ and $\Delta$ denotes the $h$-difference, i.e., $y^{\Delta}(t)=\Delta_h[y](t)=\frac{y(t+h)-y(t)}{h}$),
then
there exist the following identities
$$
\sum_{k=1}^{n} \sum_{i=0}^{m} (-1)^i
\left[
g^{k}_{ij}(t+h)\cdot E_k(L)(t,y(t+h),\Delta_h[y](t)) \right]^{\Delta_h^i}
= 0, \quad \quad t\in[a,b-mh],
\quad j=1,2,\ldots, r.
$$
\end{corollary}

For $h=1$ we obtain the analogue of Noether's second theorem for the difference calculus of variations
recently proved in \cite{Hydon+Mansfield:2011}.
In the case $\mathbb{T}=q^{\mathbb{N}_0}$ we obtain the new result.

\begin{corollary}
Let $q>1$, $L: q^{\mathbb{N}_0}\times \mathbb{R}^n \times \mathbb{R}^n
\rightarrow \mathbb{R}$ and $a,b \in q^{\mathbb{N}_0}$, $a<b$. If functional $\mathcal{L}$ defined by
$$
\mathcal{L}[y]=\sum_{t=a}^{\frac{b}{q}} L(t,y(qt),\Delta_q[y](t))
$$
is invariant under transformations (\ref{without transformations of time}) (where $\sigma$ denotes in this context the function $\sigma(t)=qt$ and $\Delta$ denotes the $q$-derivative, i.e. $y^{\Delta}(t)=\Delta_q[y](t)=\frac{y(qt)-y(t)}{(q-1)t}$),
then
there exist the following identities
$$
\sum_{k=1}^{n} \sum_{i=0}^{m} (-1)^i \left(\frac{1}{q}\right)^{\frac{i(i+1)}{2}}
\left[
g^{k}_{ij}(qt)\cdot E_k(L)(t,y(qt),\Delta_q[y](t)) \right]^{\Delta_q^i}
= 0, \quad \quad t\in\left[a,\frac{b}{q^m}\right],
\quad j=1,2,\ldots, r.
$$
\end{corollary}

In order to  prove the second Noether theorem with transformation of time,
we shall consider
that the Lagrangian $L$ is defined for all $t \in \mathbb{R}$ (not only for $t$ from the initial time scale $\mathbb{T}$),
$L: \mathbb{R}\times \mathbb{R}^n \times \mathbb{R}^n
\rightarrow \mathbb{R}$. Consider
the following transformations that depend on arbitrary functions $p_1, p_2, \ldots, p_r$ and their delta-derivatives up to order $m$:

\begin{equation}\label{with transformations of time}
\left\{
\begin{array}{ccl}
\overline{t}  & = & t + \displaystyle \sum_{j=1}^{r}H^{j}(p_{j})(t),\\
& & \\
\overline{y}_k(\overline{t})& = & y_k(t) + \displaystyle \sum_{j=1}^{r}T^{kj}(p_{j})(t), \quad k= 1,2,\ldots, n
\end{array}
\right.
\end{equation}
\\

where, for each $j=1,2,\ldots, r$, 
$p_{j} \in C^{2m}_{rd}([a,\sigma^{m}(b)], \mathbb{R})$,
$$
H^{j}(p_{j}):= \sum_{i=0}^m f_{ij}p^{\sigma^{m-(i+1)} \Delta^i}_{j}
 \quad \quad \mbox{and} \quad \quad
T^{kj}(p_{j}):= \sum_{i=0}^m g^{k}_{ij}p^{\sigma^{m-(i+1)} \Delta^i}_{j}
$$
$f_{ij} \in C^{1}_{rd}([a,b], \mathbb{R})$ \mbox{and} $g^{k}_{ij} \in C^{1}_{rd}([a,b], \mathbb{R})$.

Moreover we assume that the map
$$t\mapsto \alpha(t):= t + \displaystyle \sum_{j=1}^{r}H^{j}(p_{j})(t)$$
is a strictly increasing $C^1_{rd}$ function and its image is again a time scale, $\overline{\mathbb{T}}$.
We denote the forward shift operator relative to $\overline{\mathbb{T}}$ by $\overline{\sigma}$ and the delta derivative by $\overline{\Delta}$.
We remark that the following holds \cite{ahl}:
$$\overline{\sigma}\circ \alpha= \alpha\circ \sigma.$$

\begin{definition}\label{invariance a-b}
Functional $\mathcal{L}$ is invariant under transformations (\ref{with transformations of time})
if, and only if, for all $y\in C^1_{rd}([a,b],\mathbb{R})$ we have
$$
\int_a^b L(t,y^\sigma(t),y^\Delta(t)) \Delta t = \int_{\overline{a}}^{\overline{b}} L(\overline{t},\overline{y}^{\overline{\sigma}}(\overline{t}),\overline{y}^{\overline{\Delta}}(\overline{t})) \overline{\Delta} \overline{t}.
$$
\end{definition}

We recall the following results that will be very useful
in the proof of Theorem \ref{Theorem with transformation of time}.

\begin{theorem}[\cite{Bohner-Peterson1}]
\label{theorems}
Assume that $\nu:\mathbb{T}\rightarrow \mathbb{R}$ is strictly increasing and
$\widetilde{\mathbb{T}}:=\nu(\mathbb{T})$ is a time scale.
\begin{enumerate}
\item  (Chain rule) Let $\omega: \widetilde{\mathbb{T}}
\rightarrow \mathbb{R}$. If $\nu^\Delta(t)$ and
$\omega^{\widetilde{\Delta}}(\nu (t))$ exist for all
$t \in \mathbb{T}^\kappa$, then
$$(\omega\circ\nu)^\Delta=(\omega^{\widetilde{\Delta}}
\circ\nu)\nu^\Delta.$$
\item (Substitution in the integral)  If
$f:\widetilde{\mathbb{T}}\rightarrow \mathbb{R}$ is a $\mathrm{C}_{rd}$ function and
$\nu$ is a $\mathrm{C}^1_{rd}$ function, then for $a,b \in
\mathbb{T}$,
$$
\int_{a}^{b} f(\nu(t))\nu^\Delta(t)\Delta t =
\int_{\nu(a)}^{\nu(b)} f(s)\widetilde{{\Delta}}s.
$$
\end{enumerate}
\end{theorem}

Now we are ready to state and prove Noether's second theorem with transformation of time.

\begin{theorem}[Noether's second theorem with transformation of time]\label{Theorem with transformation of time}
If functional $\mathcal{L}$ is invariant under transformations (\ref{with transformations of time}), then
there exist the following identities
\begin{equation}\label{formula with transformation of time}
\sum_{k=1}^{n} \sum_{i=0}^{m} (-1)^i \left( \frac{1}{b_1} \right)^{\frac{i(i+1)}{2}}
\left[
 \left((g^{k}_{ij})^{\sigma} E_k(L) \right)^{\Delta^i}
+
 \left((f_{ij})^{\sigma} \left( \frac{\partial L}{\partial t} -
\frac{\Delta}{\Delta t}\left(L-y_k^\Delta \frac{\partial L}{\partial y_k^\Delta}- \mu \frac{\partial L}{\partial t}\right)
\right)\right)^{\Delta^i}\right]
\equiv 0
\end{equation}
for $j=1,2,\ldots, r$.
\end{theorem}

\begin{proof}
The idea of the proof is to reduce the statement of this result to the one of Theorem \ref{Theorem without transformation of time}
using a technique of reparametrization of time:
artificially we will consider $t$ as a dependent variable of the same footing with $y$.

Let $r\neq 0$ and define

$$
\widetilde{L}(t,s,y,r,v):=L(s-\mu(t)r,y,\frac{v}{r})r.
$$

Note that, for $s(t)=t$ and any $y\in C^1_{rd}([a,b],\mathbb{R}^n)$, we have

$$
L(t,y^\sigma(t),y^\Delta(t))= \widetilde{L}(t, s^\sigma(t),y^\sigma(t), s^\Delta(t),y^\Delta(t)).
$$
Therefore, for $s(t)=t$,
$$
\mathcal{L}[y]:=\int_a^b L(t,y^\sigma(t),y^\Delta(t)  \Delta t= \int_a^b \widetilde{L}(t, s^\sigma(t),y^\sigma(t), s^\Delta(t),y^\Delta(t))\Delta t:=
\widetilde{\mathcal{L}}[s,y].
$$

Note that, for $s(t)=t$,

\begin{equation}\label{equality lagrangians}
\begin{split}
\widetilde{\mathcal{L}}[s(\cdot),y(\cdot)]= \mathcal{L}[y(\cdot)]=
\int_{a}^{b} & L(t,y^{\sigma}(t), y^{\Delta}(t))\Delta t\\
&= \int_{\alpha(a)}^{\alpha(b)} L(\overline{t},(\overline{y}
\circ \overline{\sigma})(\overline{t}),\overline{y}^{\overline{\Delta}}(\overline{t}))
\overline{\Delta}\overline{t}\\
&= \int_{a}^{b} L\left(\alpha(t),\left(\overline{y}
\circ \overline{\sigma}\circ \alpha\right) (t),
\overline{y}^{\overline{\Delta}}(\alpha(t))\right) \alpha^{\Delta}(t) \Delta t\\
&= \int_{a}^{b} L\left(\alpha(t),\left(\overline{y}\circ \alpha \circ
\sigma\right)(t), \frac{(\overline{y} \circ \alpha)^{\Delta}(t)}{\alpha^{\Delta}(t)}\right)
\alpha^{\Delta}(t) \Delta t\\
&= \int_{a}^{b} L\left(\alpha^\sigma(t)-\mu(t)\alpha^\Delta(t),(\overline{y}\circ \alpha)^\sigma (t),\frac{(\overline{y} \circ \alpha)^{\Delta}(t)}{\alpha^{\Delta}(t)}\right)
\alpha^{\Delta}(t) \Delta t\\
&= \int_{a}^{b} \widetilde{L}\left(t,\alpha^\sigma(t),(\overline{y}\circ \alpha)^\sigma (t),\alpha^\Delta(t),(\overline{y}\circ \alpha)^\Delta (t)\right) \Delta t\\
&= \widetilde{\mathcal{L}}[\alpha(\cdot),(\overline{y}\circ \alpha)(\cdot)].
\end{split}
\end{equation}

Let $H(t,y(t)):=\alpha(t)$ and $T=(T^1, T^2, \ldots, T^n)$ where
$$T^k(t,y(t)):=y_k(t) + \displaystyle \sum_{j=1}^{r}T^{kj}(p_{j})(t), \quad k=1,2,\ldots,n.$$

Then, for $s(t)=t$,

\begin{equation}\label{equality}
(\alpha(t), (\overline{y}\circ \alpha)(t))=(\overline{t}, \overline{y}(\overline{t}))=(H(t,y(t)),T(t,y(t)))=(H(s(t),y(t)),T(s(t),y(t))).
\end{equation}

Hence, using (\ref{equality lagrangians}) and (\ref{equality}) we get
$$
\widetilde{\mathcal{L}}[s(\cdot),y(\cdot)]= \widetilde{\mathcal{L}}[H(s(\cdot),y(\cdot)),T(s(\cdot),y(\cdot))].
$$

This means that $\widetilde{\mathcal{L}}$ is invariant on
$$
\widetilde{U}=\{(s,y) | s(t)=t \wedge y \in C^1_{rd}([a,b],\mathbb{R}^n)\}
$$
under the group of state transformations
$$
(\overline{s},\overline{y})= (H(s,y), T(s,y))
$$
in the sense of Definition \ref{invariance y}.

Using Theorem \ref{Theorem without transformation of time} we can conclude that there exist the following $r$ identities $(j=1,2,\ldots, r)$
\begin{equation}\label{equality1}
\sum_{k=1}^{n} \sum_{i=0}^{m} (-1)^i \left( \frac{1}{b_1} \right)^{\frac{i(i+1)}{2}} \left((g^{k}_{ij})^{\sigma} E_k(\widetilde{L}) \right)^{\Delta^i}
+
\sum_{i=0}^{m} (-1)^i \left( \frac{1}{b_1} \right)^{\frac{i(i+1)}{2}} \left((f_{ij})^{\sigma} E_s(\widetilde{L}) \right)^{\Delta^i}
\equiv 0
\end{equation}
where we denote $E_s(\widetilde{L}):=  \frac{\partial \widetilde{L}}{\partial s^\sigma} - \frac{\Delta}{\Delta t}\frac{\partial \widetilde{L}}{\partial s^\Delta}$.

Note that, for $s(t)=t$,
$$
\frac{\partial \widetilde{L}}{\partial s^\sigma}(t, s^\sigma(t),y^\sigma(t), s^\Delta(t),y^\Delta(t))=
\frac{\partial L}{\partial t} \left(s^\sigma(t)-\mu(t)s^\Delta(t),y^\sigma(t),\frac{y^{\Delta}(t)}{s^{\Delta}(t)}\right)s^\Delta(t)
$$
and
\begin{equation*}
\begin{split}
\frac{\partial \widetilde{L}}{\partial s^\Delta}(t, s^\sigma(t),y^\sigma(t), s^\Delta(t),y^\Delta(t))
& =
L\left(s^\sigma(t)-\mu(t)s^\Delta(t),y^\sigma(t),\frac{y^{\Delta}(t)}{s^{\Delta}(t)}\right)\\
& -
\sum_{k=1}^{n} \frac{y_k^{\Delta}(t)}{s^{\Delta}(t)} \frac{\partial L}{\partial y_k^\Delta}
\left(s^\sigma(t)-\mu(t)s^\Delta(t),y^\sigma(t),\frac{y^{\Delta}(t)}{s^{\Delta}(t)}\right)\\
& -
\frac{\partial L}{\partial t}\left(s^\sigma(t)-\mu(t)s^\Delta(t),y^\sigma(t),\frac{y^{\Delta}(t)}{s^{\Delta}(t)}\right) \mu(t)s^\Delta(t).
\end{split}
\end{equation*}

Hence,  for $s(t)=t$,
\begin{equation*}
\begin{split}
&E_s(\widetilde{L})(t, s^\sigma(t),y^\sigma(t), s^\Delta(t),y^\Delta(t))\\
& = \frac{\partial L}{\partial t}(t,y^\sigma(t),y^\Delta(t)) \\
& - \frac{\Delta}{\Delta t}\left(L(t,y^\sigma(t),y^\Delta(t))-\sum_{k=1}^{n} y_k^{\Delta}(t) \frac{\partial L}{\partial y_k^\Delta} (t,y^\sigma(t),y^\Delta(t))- \mu(t) \frac{\partial L}{\partial t}(t,y^\sigma(t),y^\Delta(t))\right).
\end{split}
\end{equation*}
Also note that, for $s(t)=t$ and $k=1,2,\ldots,n$,
$$E_k(\widetilde{L})(t, s^\sigma(t),y^\sigma(t), s^\Delta(t),y^\Delta(t))=E_k(L)(t,y^\sigma(t),y^\Delta(t)).$$

Substituting the above equalities into (\ref{equality1}), we conclude the desired result:
$$
\sum_{k=1}^{n} \sum_{i=0}^{m} (-1)^i \left( \frac{1}{b_1} \right)^{\frac{i(i+1)}{2}}
\left[
 \Big((g^{k}_{ij})^{\sigma} E_k(L) \Big)^{\Delta^i}
+
 \left((f_{ij})^{\sigma} \left( \frac{\partial L}{\partial t} -
\frac{\Delta}{\Delta t}\left(L-y_k^\Delta \frac{\partial L}{\partial y_k^\Delta}- \mu \frac{\partial L}{\partial t}\right)
\right)\right)^{\Delta^i}\right]
\equiv 0
$$
for $j=1,2,\ldots, r$.
\end{proof}

\begin{remark}
Define
$$E_k^s(L):= \frac{\partial L}{\partial t} - \frac{\Delta}{\Delta t}\left(L- \sum_{k=1}^{n} y_k^\Delta \frac{\partial L}{\partial y_k^\Delta}- \mu \frac{\partial L}{\partial t}\right),$$
for $k=1,2,\ldots,n$. Then $E_k^s(L)=0$ are the second Euler--Lagrange equations for problem \eqref{problem} \cite{Zbig+Nat+Torres2011}. Therefore, expression \eqref{formula with transformation of time} provides ``dependencies'' between two types of the Euler--Lagrange expressions.
\end{remark}

Note that if $\mathbb{T}=\mathbb{R}$, Noether's identity (\ref{formula with transformation of time}) simplifies because
$$ \frac{\partial L}{\partial t} - \frac{\Delta}{\Delta t}\left(L- \sum_{k=1}^{n} y_k^\Delta \frac{\partial L}{\partial y_k^\Delta}- \mu \frac{\partial L}{\partial t}\right) =
\frac{\partial L}{\partial t} - \frac{d}{dt}\left(L- \sum_{k=1}^{n} \dot{y}_k \frac{\partial L}{\partial \dot{y}_k}\right) =
- \sum_{k=1}^{n}\dot{y}_k E_k(L)$$
and we obtain the following corollary.

%
%

\begin{corollary} [Classical  Noether's second theorem, cf. \cite{Noether1918}] If functional $\mathcal{L}$ defined by
$$
\mathcal{L}[y]=\int_a^b L(t,y(t),\dot{y}(t))\, dt
$$
is invariant under transformations (\ref{with transformations of time}) (where $\sigma$ denotes in this context the identity function and $\Delta$ denotes the usual derivative),
then there exist the following identities
$$
\sum_{k=1}^{n} \sum_{i=0}^{m} (-1)^i
\left[
 \left(g^{k}_{ij}E_k(L)\right)^{(i)} - \left(f_{ij} \cdot \dot{y}_k E_k(L)\right)^{(i)} \right]
\equiv 0,
\quad\quad j=1,2,\ldots, r.
$$
\end{corollary}

In the case $\mathbb{T}=h \mathbb{Z}$ (for some $h>0$) we obtain from Theorem \ref{Theorem with transformation of time} the second Noether theorem for the $h$-calculus; whereas if
$\mathbb{T}=q^{\mathbb{N}_0}$ (for some $q>1$) we get the second Noether theorem for the $q$-calculus.


\subsection{Noether's second theorem - multiple delta integral case}
\label{Second Noether's Theorem - multiple integral case}

In this subsection we extend the second Noether theorem (without transformation of time) to multiple integral variational problems in the time scale setting. For simplicity of presentation we prove the result for the case of two independent variables and transformations that depend on an arbitrary function and its first-order partial delta derivatives. Clearly, our result can be generalized for $n$ independent variables and $r$ arbitrary functions and their higher-order partial delta derivatives.

For the convenience of the reader we recall notions and results that are needed in the sequel. A general introduction to differential calculus and integration theory for multi-variable functions on time scales is presented, respectively, in \cite{Bohner+Guseinov2004} (see also \cite{Sarikaya2009}) and \cite{Bohner+Guseinov2005}. For the  double integral calculus of variations on time scales we refer the reader to
\cite{Bohner+Guseinov2007}.

\vspace{0.3 cm}
Let $\mathbb{T}_1$ and $\mathbb{T}_2$ be two given time scales. For $i=1,2$, denote by $\sigma_i$ and $\Delta_i$ the
forward jump operator and the delta derivative on $\mathbb{T}_i$, respectively.
Let $C^{(1)}_{rd}$ denote the set of all continuous functions defined on $\mathbb{T}_1 \times \mathbb{T}_2$ for which
both the $\Delta_1$-partial derivative and the $\Delta_2$-partial derivative exist and are of class $C_{rd}$ (for a definition see \cite{Bohner+Guseinov2007}).

Let $\Omega \subseteq \mathbb{T}_1 \times \mathbb{T}_2$ be an $\omega$-type set and let
$\Gamma$ be its positively fence (see \cite{Bohner+Guseinov2007}).
Denote
$$\Omega^{\circ}:=\{(x,y) \in \Omega: (\sigma_1(x), \sigma_2(y))\in \Omega\}.$$
Let a function $L(x,y,u,p,q)$, where $(x,y) \in \Omega\cup\Gamma$ and  $(u,p,q)\in \mathbb{R}^{3n}$ be given.
We will suppose that $L$ is continuous, together with its partial delta derivatives of first and second order with respect to $x,y$
and partial usual derivatives of the first and second order with respect to $u,p,q$. In what follows  $u^{\Delta_1}$ and $u^{\Delta_2}$ denote, respectively, $\frac{\partial u}{\Delta_1 x}$ and  $\frac{\partial u}{\Delta_2 y}$.

Consider the following optimization problem:

\begin{equation}
\label{problem2}
\mathcal{L}[u]=\int \int_\Omega L(x,y,u(\sigma_1(x), \sigma_2(y)),u^{\Delta_1}(x, \sigma_2(y)), u^{\Delta_2}(\sigma_1(x), y) ) \Delta_1 x\Delta_2 y
\longrightarrow {\rm extremize}
\end{equation}
where the set of admissible functions are
$$
\mathcal{D}=\{ u\ | \ u: \Omega\cup \Gamma
\rightarrow \mathbb{R}^n,\ u\in C^{(1)}_{rd},\
u=g \ \mbox{on} \  \Gamma\}
$$
where $g$ is a fixed function defined and continuous on the fence $\Gamma$ of $\Omega$.

As noticed in \cite{Bohner+Guseinov2007}, for the variational problem (\ref{problem2}) be well posed, we have to assume that there exists at least one admissible function $u_0 \in \mathcal{D}$ because it is possible to choose a continuous function $g$ such that no function $u$ is admissible.
Note that if there exists an admissible function $u_0$, then the set $\mathcal{D}$ contains a set of functions of the form $u=u_0 + \eta$, where $\eta:\Omega\cup \Gamma \rightarrow\mathbb{R}^n$ is  $C^{(1)}_{rd}$ and $\eta=0$ on $\Gamma$.
Any such $\eta$ is called an admissible variation for problem (\ref{problem2}).

\begin{definition}
We say that $u_{\ast}\in \mathcal{D}$
is a local minimizer (resp. local maximizer) for problem  (\ref{problem2})
if there exists $\delta > 0$ such that
$$
\mathcal{L}[u_{\ast}]\leq \mathcal{L}[u] \quad ( {\rm resp.}  \ \mathcal{L}[u_{\ast}]\geq \mathcal{L}[u])
$$
for all $u \in \mathcal{D}$ with
$$
\parallel u - u_{\ast}\parallel :=
\sup_{(x,y) \in \Omega \cup \Gamma}\mid u(x,y)-u_{\ast}(x,y)\mid
+ \sup_{(x,y) \in \Omega}\mid u^{\Delta_1}(x,\sigma_2(y))-u_{\ast}^{\Delta_1}(x,\sigma_2(y))\mid$$
$$  +
\sup_{(x,y) \in \Omega}\mid u^{\Delta_2}(\sigma_1(x),y)-u_{\ast}^{\Delta_2}(\sigma_1(x),y)\mid
<\delta \, ,
$$
where $|\cdot|$ denotes a norm in $\mathbb{R}^n$.
\end{definition}

We recall the following results which will play an important role in the proofs of our results.

\begin{theorem}[Green's Theorem, \cite{Bohner+GuseinovGREEN2007}]
If the functions $M$ and $N$ are continuous and have continuous partial delta derivatives $\frac{\partial M}{\Delta_2 y}$
and $\frac{\partial N}{\Delta_1 x}$ on $\Omega \cup \Gamma$, then
\begin{equation} \label{Green}
\int \int_\Omega \left( \frac{\partial N}{\Delta_1 x} - \frac{\partial M}{\Delta_2 y} \right) \Delta_1x\Delta_2y = \int_{\Gamma} M d^\ast x + N d^\ast y,
\end{equation}
where the "star line integrals" on the right side in (\ref{Green}) denote the sum of line delta integrals taken over the line segment constituents of
$\Gamma$ directed to the right or upwards and line nabla integrals taken over the line segment constituents of
$\Gamma$ directed to the left or downwards.
\end{theorem}

\begin{lemma} [Fundamental lemma of the double variational calculus, \cite{Bohner+Guseinov2007}]
\label{Fundamental lemma of the double variational calculus}
If $M$ is continuous on $\Omega\cup \Gamma$ with
$$
\int \int_\Omega M(x,y) \eta(\sigma_1(x), \sigma_2(y)) \Delta_1 x\Delta_2 y=0
$$
for any admissible variation $\eta$, then
$$
M(x,y)=0 \ \mbox{for all} \  (x,y)\in \Omega^{\circ}.
$$
\end{lemma}

\begin{theorem}[Euler--Lagrange equation of the double variational calculus,  \cite{Bohner+Guseinov2007}]
Suppose that an admissible function $u_\ast$ provides a local minimum for $\mathcal{L}$ and that $u_\ast$ has continuous partial delta derivatives of the second order.
Then $u_\ast$ satisfies the Euler--Lagrange equation
$$
\frac{\partial L}{\partial u}(\cdot)
- \frac{\partial}{\Delta_1 x} \frac{\partial L}{\partial p}(\cdot)
- \frac{\partial}{\Delta_2 y} \frac{\partial L}{\partial q}(\cdot)=0
$$
where
$
(\cdot)=(x,y,u(\sigma_1(x), \sigma_2(y)),u^{\Delta_1}(x, \sigma_2(y)), u^{\Delta_2}(\sigma_1(x), y))
$
for $(x,y) \in \Omega^{\circ}$.
\end{theorem}

Let us denote by $\rho_1$ and $\rho_2$ the backward jump operator of $\mathbb{T}_1$ and $\mathbb{T}_2$, respectively.
In what follows we will suppose that $\mathbb{T}_1$ and $\mathbb{T}_2$ are such
\begin{equation}\label{property1}
\sigma_1(\rho_1(x))=x, \quad \forall x \in {(\mathbb{T}_1)}_{\kappa}
\end{equation}
and
$$
\sigma_2(\rho_2(y))=y, \quad \forall y \in {(\mathbb{T}_2)}_{\kappa},
$$

where $\mathbb{T}_{\kappa}:=\mathbb{T}\setminus\{m\}$ if $\mathbb{T}$ has a right-scattered minimum $m$; otherwise, $\mathbb{T}_{\kappa}=\mathbb{T}$.
We recall the fact that $\mathbb{R}$, $h \mathbb{Z}$ (for some $h>0$), $q^{\mathbb{N}_0}$ (for some $q>1$)
and many other interesting time scales satisfy property (\ref{property1}).

Let $u(x,y)=(u_1(x,y), u_2(x,y), \ldots, u_n(x,y))$ and consider the following transformations that depend on an arbitrary continuous function $p$
and the partial delta derivatives of $p$:

\begin{equation}\label{without transformations of time double integral}
\left\{
\begin{array}{ccl}
\overline{x}  & = & x\\
& & \\
\overline{y}  & = & y\\
& & \\
\overline{u}_k(\overline{x}, \overline{y})& = & u_k(x,y) + T^k(p)(x,y)
\end{array}
\right.
\end{equation}
\\

where, for each $k= 1,2,\ldots, n$,
$$T^k(p)(x,y):=  a_0^k(x,y)p(x,y) +
a_1^k(x,y)\frac{\partial}{\Delta_1 x}p(\rho_1(x),y) +
a_2^k(x,y)\frac{\partial}{\Delta_2 y}p(x, \rho_2(y)), $$
 $a_0, a_1, a_2$ are  $C^{1}$ functions  and we assume that $p$ has continuous partial delta derivatives of the first and second order.

\begin{definition}\label{invariance double integral}
Functional $\mathcal{L}$ is invariant under transformations (\ref{without transformations of time double integral})
if, and only if, for all $u \in \mathcal{D}$ we have
$$
\int \int_\Omega L(x,y,u(\sigma_1(x), \sigma_2(y)),u^{\Delta_1}(x, \sigma_2(y)), u^{\Delta_2}(\sigma_1(x), y) ) \Delta_1 x\Delta_2 y
$$
$$=
\int \int_\Omega L(x,y,\overline{u}(\sigma_1(x), \sigma_2(y)),\overline{u}^{\Delta_1}(x, \sigma_2(y)), \overline{u}^{\Delta_2}(\sigma_1(x), y) ) \Delta_1 x\Delta_2 y.
$$
\end{definition}

In what follows we use the notations
$$T^k(p^{\sigma})(x,y):=T^k(\sigma_1(x),\sigma_2(y)), \ T^k(p^{\sigma_1})(x,y):=T^k(\sigma_1(x),y), \ T^k(p^{\sigma_2})(x,y):=T^k(x,\sigma_2(y)).$$

Using similar arguments as the ones used in the proof of Theorem \ref{necessary condition for invariance} we can prove the following result.

\begin{theorem}[Necessary condition of invariance]\label{necessary condition for invariance double integral}
If functional $\mathcal{L}$ is invariant under transformations (\ref{without transformations of time double integral}), then
\begin{equation}\label{invariance2}
\sum_{k=1}^{n} \int \int_\Omega  \left( \frac{\partial L}{\partial u^{\sigma}_k} \cdot T^k(p^{\sigma}) +
\frac{\partial L}{\partial u^{\Delta_1}_{k}}\frac{\partial}{\Delta_1 x}T^k(p^{\sigma_2})+
\frac{\partial L}{\partial u^{\Delta_2}_{k}}\frac{\partial}{\Delta_2 y}T^k(p^{\sigma_1})
\right) \Delta_1 x\Delta_2 y=0.
\end{equation}
\end{theorem}

Define

$$
\widehat{E}_k(L):=
\frac{\partial L}{\partial u^{\sigma}_k}
- \frac{\partial}{\Delta_1 x} \frac{\partial L}{\partial u^{\Delta_1}_{k}}
- \frac{\partial}{\Delta_2 y} \frac{\partial L}{\partial u^{\Delta_2}_{k}},  \quad k=1,2, \ldots, n.
$$
We call $\widehat{E}_k(L)$,  $k=1,2, \ldots, n$,  the Euler--Lagrange expressions associated to the Lagrangian $L$ relative to problem
(\ref{problem2}).

The following lemmas will be used in the proof of Theorem \ref{Noether's second theorem without transforming time double integral}.

\begin{lemma} \label{Invariant double integral}
If $\mathcal{L}$ is invariant under transformations (\ref{without transformations of time double integral}), then
$$
\sum_{k=1}^{n} \int \int_\Omega \widehat{E}_{k}(L) \cdot T^{k}(p^{\sigma}) \Delta_1 x\Delta_2 y=0.
$$
\end{lemma}

\begin{proof}
Fix $i \in \{1,2,\ldots,n\}$. Observe that
$$
\int \int_\Omega  \left(  \frac{\partial L}{\partial u^{\Delta_1}_{i}}\frac{\partial}{\Delta_1 x}T^i(p^{\sigma_2})+
\frac{\partial L}{\partial u^{\Delta_2}_{i}}\frac{\partial}{\Delta_2 y}T^i(p^{\sigma_1})
\right) \Delta_1 x\Delta_2 y
$$
$$
= \int \int_\Omega \left[ \frac{\partial}{\Delta_1 x} \left(  \frac{\partial L}{\partial u^{\Delta_1}_{i}} \cdot T^i(p^{\sigma_2}) \right) +
\frac{\partial}{\Delta_2 y} \left(  \frac{\partial L}{\partial u^{\Delta_2}_{i}} \cdot T^i(p^{\sigma_1}) \right)\right] \Delta_1 x\Delta_2 y
$$
$$
- \int \int_\Omega  \left[  \frac{\partial}{\Delta_1 x} \frac{\partial L}{\partial u^{\Delta_1}_{i}} \cdot T^i(p^{\sigma})+
 \frac{\partial}{\Delta_2 y} \frac{\partial L}{\partial u^{\Delta_2}_{i}} \cdot T^i(p^{\sigma})
\right] \Delta_1 x\Delta_2 y.
$$

Using Green's Theorem we get,
$$
\int \int_\Omega \left[ \frac{\partial}{\Delta_1 x} \left(  \frac{\partial L}{\partial u^{\Delta_1}_{i}} \cdot T^i(p^{\sigma_2}) \right) +
\frac{\partial}{\Delta_2 y} \left(  \frac{\partial L}{\partial u^{\Delta_2}_{i}} \cdot T^i(p^{\sigma_1}) \right)\right] \Delta_1 x\Delta_2 y$$
$$
= \int_{\Gamma} \frac{\partial L}{\partial u^{\Delta_1}_{i}} \cdot T^i(p^{\sigma_2})  d^\ast y -   \frac{\partial L}{\partial u^{\Delta_2}_{i}} \cdot T^i(p^{\sigma_1}) d^\ast x.
$$
Since $p$ is arbitrary we can choose $p$ such that
$$
p(x,\sigma_2(y))|_{\Gamma}=0, \quad \quad p(\sigma_1(x),y)|_{\Gamma}=0,
$$
$$
\frac{\partial}{\Delta_1 x} p(\rho_1(x), \sigma_2(y))|_{\Gamma}=0, \quad \quad
\frac{\partial}{\Delta_1 x} p(x, y)|_{\Gamma}=0,
$$
$$\frac{\partial}{\Delta_2 y} p(x, y)|_{\Gamma}=0, \quad \quad
\frac{\partial}{\Delta_2 y} p(\sigma_1(x), \rho_2(y))|_{\Gamma}=0,
$$
and therefore
$$
\int \int_\Omega \left[ \frac{\partial}{\Delta_1 x} \left(  \frac{\partial L}{\partial u^{\Delta_1}_{i}} \cdot T^i(p^{\sigma_2}) \right) +
\frac{\partial}{\Delta_2 y} \left(  \frac{\partial L}{\partial u^{\Delta_2}_{i}} \cdot T^i(p^{\sigma_1}) \right)\right] \Delta_1 x\Delta_2 y=0.
$$
By Theorem \ref{necessary condition for invariance double integral} we obtain
$$
\sum_{k=1}^{n} \int \int_\Omega  \left( \frac{\partial L}{\partial u^{\sigma}_k} \cdot T^k(p^{\sigma})
- \frac{\partial}{\Delta_1 x} \frac{\partial L}{\partial u^{\Delta_1}_{k}} \cdot T^k(p^{\sigma})
- \frac{\partial}{\Delta_2 y} \frac{\partial L}{\partial u^{\Delta_2}_{k}} \cdot T^k(p^{\sigma})\right)\Delta_1 x\Delta_2 y=0
$$
which proves that
$
\displaystyle \sum_{k=1}^{n} \int \int_\Omega \widehat{E}_{k}(L) \cdot T^k(p^{\sigma}) \Delta_1x \Delta_2 y=0.
$
\end{proof}

\begin{lemma}\label{adjointoperator}
For each $k=1,2,\ldots,n$,
$$ \int \int_\Omega q \cdot T^{k}(p^{\sigma}) \Delta_1x \Delta_2y =
\int \int_\Omega \left( q a_0^{k} - \frac{\partial}{\Delta_1 x}(q a_1^{k}) - \frac{\partial}{\Delta_2 y} (q a_2^{k}) \right) \cdot p^{\sigma} \Delta_1 x \Delta_2y
$$ holds.
\end{lemma}

\begin{proof}
Note that
\begin{equation*}
\begin{split}
 \int \int_\Omega q \cdot T^{k}(p^{\sigma}) \Delta_1x \Delta_2 y&=\int \int_\Omega \left[ q a_0^k(x,y)p^\sigma(x,y)\right.
 + q a_1^k(x,y)\frac{\partial}{\Delta_1 x}p(x,\sigma_2(y)) \\
 & \left.  \quad \quad \quad +
 q a_2^k(x,y)\frac{\partial}{\Delta_2 y}p(\sigma_1(x),y)\right] \Delta_1 x \Delta_2 y
\end{split}
\end{equation*}

and
\begin{equation*}
\begin{array}{l}
\displaystyle \int \int_\Omega \left[ q a_1^k(x,y)\frac{\partial}{\Delta_1 x}p(x,\sigma_2(y)) + q a_2^k(x,y)\frac{\partial}{\Delta_2 y}p(\sigma_1(x),y)\right] \Delta_1 x \Delta_2 y\\
 = \displaystyle \int \int_\Omega  \left[\frac{\partial}{\Delta_1 x}\left(q a_1^k(x,y) p(x,\sigma_2(y)) \right)+
\frac{\partial}{\Delta_2 y}\left(q a_2^k(x,y) p(\sigma_1(x),y) \right)\right] \Delta_1 x \Delta_2 y \\
- \displaystyle \int \int_\Omega \left[\frac{\partial}{\Delta_1 x}(q a_1^k(x,y))\cdot p(\sigma_1(x),\sigma_2(y)) +
\frac{\partial}{\Delta_2 y}(q a_2^k(x,y))\cdot p(\sigma_1(x),\sigma_2(y))\right] \Delta_1 x \Delta_2 y.
\end{array}
\end{equation*}

Using  Green's Theorem we can conclude that
$$
\int \int_\Omega  \left[\frac{\partial}{\Delta_1 x}\left(q a_1^k(x,y) p(x,\sigma_2(y)) \right)+
\frac{\partial}{\Delta_2 y}\left(q a_2^k(x,y) p(\sigma_1(x),y) \right)\right] \Delta_1 x \Delta_2 y=0.$$

Hence
$$
\int \int_\Omega q \cdot T^{k}(p^{\sigma}) \Delta_1 x\Delta_2y
=
 \int \int_\Omega \left[ q a_0^k \cdot p^\sigma  - \frac{\partial}{\Delta_1 x}(q a_1^k )\cdot p^\sigma -
\frac{\partial}{\Delta_2 y}(q a_2^k )\cdot p^\sigma\right] \Delta_1x  \Delta_2y
$$
proving the desired result.

\end{proof}

\begin{remark}

Lemma \ref {adjointoperator} shows that we can define an adjoint operator of $T^k$, $\widetilde{T}^k$, by $$\widetilde{T}^{k}(q)= q a_0^{k} - \frac{\partial}{\Delta_1 x}(q a_1^{k}) - \frac{\partial}{\Delta_2 y} (q a_2^{k}).$$
\end{remark}

We are now ready to state and prove the Noether second theorem without transformation of time for multiple integral problems on time scales.

\begin{theorem} [Noether's second theorem without transforming time] \label{Noether's second theorem without transforming time double integral}
If functional $\mathcal{L}$ is invariant under transformations (\ref{without transformations of time double integral}), then,
$$ \sum_{k=1}^{n} \widetilde{T}^{k} (\widehat{E}_{k}(L)) \equiv 0  \quad \mbox{on} \quad \Omega^{\circ}
$$
where $\widehat{E}_{k}(L)$ are the $n$ Euler--Lagrange expressions and $\widetilde{T}^{k}$ is the adjoint operator of $T^k$.
\end{theorem}

\begin{proof}
Using Lemma \ref{Invariant double integral} and Lemma \ref{adjointoperator} we conclude that if  $\mathcal{L}$ is invariant under transformations (\ref{without transformations of time double integral}), then
$$
\sum_{k=1}^{n} \int \int_\Omega \widehat{E}_{k}(L) \cdot T^k(p^{\sigma}) \Delta_1 x\Delta_2y =\sum_{k=1}^{n} \int \int_\Omega \widetilde{T}^k (\widehat{E}_{k}(L)) \cdot p^{\sigma} \Delta_1 x\Delta_2 y=0,$$
where  $\widetilde{T}^k$ is the adjoint operator of $T^k$.
Applying the fundamental lemma of the double variational calculus (Lemma \ref{Fundamental lemma of the double variational calculus})
we get  $$ \sum_{k=1}^{n} \widetilde{T}^{k} (\widehat{E}_{k}(L))\equiv 0  \quad \mbox{on} \quad \Omega^{\circ}$$
proving the desired result.
\end{proof}

\begin{corollary}[Classical  Noether's second theorem for double integrals problems, cf. \cite{Noether1918}]
Let $\Omega \subseteq \mathbb{R}^2$ be an $\omega$-type set and let $\Gamma$ be its positive fence.
Let $L(x,y,u,p,q)$ be a function of class $C^2$, $(x,y) \in \Omega \cup \Gamma$, $u=(u_1,u_2, \ldots, u_n)$.
If functional $\mathcal{L}$ defined by
$$
\mathcal{L}[y]=\int\int_{\Omega} L(x,y, u(x,y), \frac{\partial u}{\partial x}(x,y), \frac{\partial u}{\partial y}(x,y))\, dx dy
$$
is invariant under transformations (\ref{without transformations of time double integral}) (where $\rho_1$ and $\rho_2$ denote in this context the identity function, and $\Delta_1$ and $\Delta_2$ denote the usual derivative),
then
$$ \sum_{k=1}^{n} \widetilde{T}^{k} (\widehat{E}_{k}(L)) \equiv 0  \quad \mbox{on} \quad \Omega
$$
where
$$
\widehat{E}_k(L):=
\frac{\partial L}{\partial u_k}
- \frac{\partial}{\partial x} \frac{\partial L}{\partial p_{k}}
- \frac{\partial}{\partial y} \frac{\partial L}{\partial q_{k}}, \quad k=1,2,\ldots,n
$$
and $\widetilde{T}^{k}$ is the adjoint operator of $T^k$.
\end{corollary}

Similarly to the single delta integral case choosing $\mathbb{T}=h \mathbb{Z}$ (for some $h>0$) we obtain from Theorem \ref{Noether's second theorem without transforming time double integral}
the second Noether theorem for the double variational $h$-calculus; whereas choosing
$\mathbb{T}=q^{\mathbb{N}_0}$ (for some $q>1$) we obtain the second Noether theorem for the double variational $q$-calculus.

\section{Example}
\label{Example}

In order to illustrate the second Noether Theorem for the multiple integral case we will present the following example.
Let $\mathbb{T}_0$, $\mathbb{T}_1$, $\mathbb{T}_2$ and $\mathbb{T}_3$ be time scales and let $\Omega \subseteq \mathbb{T}_0 \times \mathbb{T}_1 \times \mathbb{T}_2\times \mathbb{T}_3$ be an $\omega$-type set.
 For $i=0,1,2,3$, denote by $\sigma_i$, $\rho_i$ and $\Delta_i$ the
forward jump operator, the backward jump operator and the delta derivative on $\mathbb{T}_i$, respectively.

Let $t:=(t_0,t_1,t_2,t_3) \in \Omega$  and consider the following real functions defined on $\Omega$: $A_0, A_1, A_2, A_3$.
Let $\mathbf{A}:=(A_1,A_2,A_3)$ and denote

\begin{equation*}
\begin{split}
\nabla A_0 (t)
:=&\left( \frac{\partial A_0}{\Delta_1 t_1} (\sigma_0(t_0), t_1, \sigma_2(t_2), \sigma_3(t_3) )\right.,\\
&\left.\frac{\partial A_0}{\Delta_2 t_2} (\sigma_0(t_0), \sigma_1(t_1), t_2, \sigma_3(t_3) ),
\frac{\partial A_0}{\Delta_3 t_3} (\sigma_0(t_0), \sigma_1(t_1), \sigma_2(t_2), t_3 )\right)
\end{split}
\end{equation*}

\begin{equation*}
\begin{split}
\frac{\partial \mathbf{A}}{\Delta_0 t_0}(t)
:=&
\left( \frac{\partial A_1}{\Delta_0 t_0} (t_0, \sigma_1(t_1), \sigma_2(t_2), \sigma_3(t_3) )\right.,\\
&\left.\frac{\partial A_2}{\Delta_0 t_0} (t_0, \sigma_1(t_1), \sigma_2(t_2), \sigma_3(t_3) ),
\frac{\partial A_3}{\Delta_0 t_0}(t_0, \sigma_1(t_1), \sigma_2(t_2), \sigma_3(t_3) )\right)
\end{split}
\end{equation*}

and

\begin{equation*}
\begin{split}
curl \mathbf{A}(t)
:= &\left( \frac{\partial A_3}{\Delta_2 t_2} (\sigma_0(t_0), \sigma_1(t_1), t_2, \sigma_3(t_3))
- \frac{\partial A_2}{\Delta_3 t_3} (\sigma_0(t_0),  \sigma_1(t_1), \sigma_2(t_2),t_3) , \right.\\
&
\frac{\partial A_1}{\Delta_3 t_3} (\sigma_0(t_0), \sigma_1(t_1), \sigma_2(t_2), t_3) -
\frac{\partial A_3}{\Delta_1 t_1} (\sigma_0(t_0), t_1, \sigma_2(t_2),\sigma_3(t_3)),\\
&
\left.\frac{\partial A_2}{\Delta_1 t_1} (\sigma_0(t_0), t_1, \sigma_2(t_2),\sigma_3(t_3))-
\frac{\partial A_1}{\Delta_2 t_2} (\sigma_0(t_0), \sigma_1(t_1), t_2, \sigma_3(t_3))
\right).
\end{split}
\end{equation*}

We will consider the following Lagrangian function

$$
L= \frac{1}{2} \left\|\nabla A_0 - \frac{\partial \mathbf{A}}{\Delta_0 t_0} \right\|^2    -   \frac{1}{2}\left \|curl \mathbf{A}\right\|^2
$$

that is the time scale version of the Lagrangian density for the electromagnetic field (see, for example, \cite{LoganBook1977}).

It can be proved that the functional

$$
\mathcal{L}=\int \cdots \int_\Omega L \, \Delta_0 \Delta_1 \Delta_2 \Delta_3
$$

is invariant under the gauge transformations

$$
\overline{A}_k=A_k+\frac{\partial}{\Delta_k t_k} p^{\rho_k} , \quad k=0,1,2,3
$$

where $p:\Omega \rightarrow R$ is an arbitrary continuous function  that has continuous partial delta derivatives of the first and second order
(hence, we have equality of mixed partial delta derivatives, see \cite{Bohner+Guseinov2004}).

Since, for each $k=0,1,2,3$, $$T^k(p)= \frac{\partial}{\Delta_k t_k} p^{\rho_k}$$  then, by Lemma \ref{adjointoperator}, we conclude that

$$\widetilde{T}^{k}(q)= - \frac{\partial}{\Delta_k t_k} q.$$

Hence, from the second Noether theorem (Theorem \ref{Noether's second theorem without transforming time double integral}), we get

$$
\sum_{k=0}^{3} \frac{\partial}{\Delta_k t_k} \widehat{E}_k(L)\equiv 0  \quad \mbox{on} \quad \Omega^{\circ},$$

where $\widehat{E}_{k}(L)$, $k=0,1,2,3$, are the Euler--Lagrange expressions associated to functional $\mathcal{L}$.

If we suppose that, for each $k=0,1,2,3$, $A_k$ has continuous partial delta derivatives of the first and second order
and that
$A_0$ and the vector field $\mathbf{A}$ satisfy the so called Lorentz conditions on time scales:

$$ div \mathbf{A}|_{(t_0, \sigma_1(t_1), \sigma_2(t_2), \sigma_3(t_3))} = \frac{\partial A_0}{\Delta_0 t_0}|_{(t_0, \sigma_1(t_1), \sigma_2(t_2), \sigma_3(t_3))}$$

$$ div \mathbf{A}|_{(\sigma_0(t_0), t_1, \sigma_2(t_2), \sigma_3(t_3))} = \frac{\partial A_0}{\Delta_0 t_0}|_{(\sigma_0(t_0), t_1, \sigma_2(t_2), \sigma_3(t_3))}$$

$$ div \mathbf{A}|_{(\sigma_0(t_0), \sigma_1(t_1), t_2, \sigma_3(t_3))} = \frac{\partial A_0}{\Delta_0 t_0}|_{(\sigma_0(t_0), \sigma_1(t_1), t_2, \sigma_3(t_3))}$$

$$ div \mathbf{A}|_{(\sigma_0(t_0), \sigma_1(t_1), \sigma_2(t_2), t_3)} = \frac{\partial A_0}{\Delta_0 t_0}|_{(\sigma_0(t_0), \sigma_1(t_1), \sigma_2(t_2), t_3)}$$

where $div \mathbf{A}$ denotes the divergence of a vector field $\mathbf{A}$, that is, $$div \mathbf{A} := \frac{\partial A_1}{\Delta_1 t_1}+ \frac{\partial A_2}{\Delta_2 t_2}+ \frac{\partial A_3}{\Delta_3 t_3},$$ then
the Euler--Lagrange expressions can be written in the following way:

$$
\widehat{E}_k(L)=   \frac{\partial^2 A_k}{\Delta_0 t_0^2} (t_0,\sigma_1(t_1),\sigma_2(t_2), \sigma_3(t_3))  - \nabla^2 A_k(t_0,t_1,t_2,t_3), \quad \quad k=0,1,2,3
$$

where

\begin{equation*}
\begin{split}
\nabla^2 A_k(t):= &\frac{\partial^2 A_k}{\Delta_1 t^2_1}(\sigma_0(t_0), t_1, \sigma_2(t_2), \sigma_3(t_3)) +
\frac{\partial^2 A_k}{\Delta_2 t^2_2}(\sigma_0(t_0), \sigma_1(t_1), t_2, \sigma_3(t_3))\\
& + \frac{\partial^2 A_k}{\Delta_3 t^2_3}(\sigma_0(t_0), \sigma_1(t_1), \sigma_2(t_2),t_3).
\end{split}
\end{equation*}

Hence, under these assumptions, we can conclude that

$$\sum_{k=0}^{3} \frac{\partial}{\Delta_k t_k} \left(\frac{\partial^2 A_k}{\Delta_0 t_0^2} (t_0,\sigma_1(t_1),\sigma_2(t_2), \sigma_3(t_3))  - \nabla^2 A_k(t_0,t_1,t_2,t_3)\right)= 0  \quad \mbox{on} \quad \Omega^{\circ}.$$


\section{Concluding remarks}
\label{Concluding Remarks}

We proved that the important Noether's second theorem is valid not only for the continuous and discrete calculus,
but also for the quantum calculus.
Moreover, in our opinion,
the proofs presented in this paper are elegant and clear to follow.

The question of obtaining  Noether's second theorem for multiple integrals with transformation of time in the time scale setting
remains an interesting open question. To the best authors' knowledge, to extend the second Noether theorem to multiple integrals with transformation of time,
substitution in the multiple integral is a fundamental tool and this result is not yet available in the literature.

For other generalizations of the second Noether theorem we refer the reader to \cite{Malinowska2Nother} (in the context of the fractional calculus of variations) and \cite{Torres-2003} (in the context of the optimal control).

\section*{Acknowledgements}

Work supported by {\it FEDER} funds through {\it COMPETE}--Operational Programme Factors of Competitiveness
(``Programa Operacional Factores de Competitividade'') and by Portuguese funds through the
{\it Center for Research and Development in Mathematics and Applications} (University of Aveiro)
and the Portuguese Foundation for Science and Technology (``FCT--Funda\c{c}\~{a}o para a Ci\^{e}ncia e a Tecnologia''),
within project PEst-C/MAT/UI4106/2011 with COMPETE number FCOMP-01-0124-FEDER-022690. A. B. Malinowska was also supported by Bialystok
University of Technology grant S/WI/02/2011.



\end{document}